\documentclass[reqno,11pt]{amsart}

\usepackage{amsmath,amsfonts,amssymb,amsthm,epsfig}

 \allowdisplaybreaks
\numberwithin{equation}{section}

\font\script=rsfs10 at 12pt
\def\eps{\varepsilon}
\def\H{{\mbox{\script H}\,\,}}

\def\F{\mathfrak F}
\def\I{\mathfrak I}
\def\R{\mathbb R}
\def\S{\mathbb S}
\def\N{\mathbb N}
\def\bal{\begin{aligned}}
\def\eal{\end{aligned}}
\def\proofof#1{\begin{proof}[Proof of #1]}
\def\Chi#1{\hbox{{\large $\chi$}{\Large $_{_{#1}}$}}}

\def\freccia#1{\xrightarrow[\ #1]{}}
\def\XXint#1#2#3{{\setbox0=\hbox{$#1{#2#3}{\int}$} \vcenter{\vspace{-1pt}\hbox{$#2#3$}}\kern-.5\wd0}}
\def\Xint#1{\mathchoice {\XXint\displaystyle\textstyle{#1}}{\XXint\textstyle\scriptstyle{#1}}{\XXint\scriptstyle\scriptscriptstyle{#1}}{\XXint\scriptscriptstyle\scriptscriptstyle{#1}}\!\int}
\def\intmed{\Xint{\hbox{---}}}

\newcounter{mt}
\def\maintheorem#1#2#3{\par \medskip \noindent {\bf Theorem~\mref{#1}}~(#2).~{\it #3}\par}
\def\mref#1{\Alph{#1}}
\def\maintheoremdeclaration#1{\stepcounter{mt}\newcounter{#1}\setcounter{#1}{\arabic{mt}}}

\maintheoremdeclaration{Main}

\newtheorem{theorem}{Theorem}[section]
\newtheorem{lemma}[theorem]{Lemma}
\newtheorem{prop}[theorem]{Proposition}

\newtheorem{defin}[theorem]{Definition}
\newtheorem{remark}[theorem]{Remark}

\begin{document}

\title{Sharp stability for the Riesz potential}

\author{N. Fusco}
\author{A. Pratelli}

\begin{abstract}
In this paper we show the stability of the ball as maximizer of the Riesz potential among sets of given volume. The stability is proved with sharp exponent $1/2$, and is valid for any dimension $N\geq 2$ and any power $1<\alpha<N$.
\end{abstract}

\maketitle

\section{Introduction}
The celebrated \emph{Riesz inequality} states that for any two positive functions $f,\,g:\R^N\to \R^+$ and any positive, decreasing function $h:\R^+\to \R^+$, one has
\begin{equation}\label{Riesz}
\int_{\R^N}\int_{\R^N} f(z)g(y) h(|y-z|)\,dy\,dz \leq \int_{\R^N}\int_{\R^N} f^*(z) g^*(y) h(|y-z|)\,dy\,dz\,,
\end{equation}
where $f^*$ and $g^*$ are the radially symmetric decreasing rearrangments of $f$ and $g$. In the special case $f=g$, with the additional assumption that $h$ is strictly decreasing, equality holds in~(\ref{Riesz}) if and only if $f=f^*$ up to a translation.

When $f$ and $g$  coincide  with the characteristic function of a set $E\subset\R^N$ of finite measure and $h(t)=t^{\alpha-N}$ for some $0<\alpha<N$, \eqref{Riesz} states that if $E$ has the same volume $\omega_N$ of the  unit ball $B=\{|x|<1\}$, then 
\begin{equation}\label{equa}
\F(E)\leq\F(B)\,,
\end{equation}
where the functional $\F$ is defined as
\begin{equation}\label{defF}
\F(E) = \int_E \int_E \frac 1{|y-x|^{N-\alpha}}\, dy\, dx\,.
\end{equation}
Moreover,  equality holds in \eqref{equa}  if and only if $E$ is a ball of radius $1$. Note that when $N=3$ and $\alpha=2$, up to a multiplicative constant, $\F(E)$ is the electrostatic energy of a uniform distributions of charges in $E$. Therefore  \eqref{equa} is easily explained by observing   that symmetrization reduces the distance between the charges, thus increasing the electrostatic repulsion between them.

Here we show the stability of the Riesz type inequality \eqref{equa}, i.e., we prove that the energy deficit $D(E)$ of the set $E$ controls a suitable distance $\delta(E)$ from $E$ to an optimal unit ball $B_x$ with center $x\in\R^N$. Precisely, setting
\begin{align*}
D(E)=\F(B)-\F(E)\,, && \delta(E) = \inf_{x\in \R^N} |B_x\Delta E|\,,
\end{align*}
where $\Delta$ denotes the symmetric difference between sets,  we show the following quantitative estimate.

\maintheorem{Main}{Sharp quantitative estimate}{Let $N\geq 2$ and $1<\alpha<N$ be given. There exists a constant $C=C(N,\alpha)$ such that, for every measurable set $E\subset\R^N$ with $|E|=\omega_N$, 
\begin{equation}\label{sharpest}
\delta(E)\leq C(N,\alpha) \sqrt{D(E)}\,.
\end{equation}}

Estimate \eqref{sharpest} was already obtained by Burchard and Chambers in \cite{BC} in the special case of  the Coulomb energy, that is when $N=3$ and $\alpha=2$. Beside, they observe that the  square root  on the right hand side is sharp, in the sense that the exponent $1/2$ cannot be replaced by any larger one. In the same paper they also prove that if $N>3$ and  $\alpha=2$, a similar inequality   holds with the exponent $1/2$ replaced by the not sharp one $1/(N+2)$. Their approach is based on a symmetrization lemma similar to the  one proved in \cite[Th. 2.1]{FMP} which allows them to reduce the proof of \eqref{sharpest} to the case of a set $E$, symmetric with respect to $N$ orthogonal hyperplanes. 

Our proof follows a different path. As for other stability estimates, such as the ones concerning the isoperimetric  and  the Faber-Krahn inequality, see \cite{CL, BDV}, the starting point here is  a Fuglede-type estimate.  More precisely, we  show with a second variation argument that
\begin{equation}\label{intro1} 
\delta(E)\leq |E\Delta B|\leq C(N,\alpha)\sqrt{D(E)}\,,
\end{equation}
whenever $E$ is \emph{nearly spherical}, that is $|E|=|B|$, $E$ has barycenter at  the origin and its boundary can be written as a graph of a function $u$ on  $\partial B$ with $\|u\|_{L^\infty(\partial B)}\ll 1$, see Proposition~\ref{fuglede}.

This first step is relatively easy. The difficult task is to show that one can always reduce the general case to the one of a nearly spherical set.
More precisely in Proposition 2.1 we show that, given a set $E$, either \eqref{sharpest} is true with a suitable constant or we can find a nearly spherical set $\widetilde E$ such that
\begin{align}\label{intro2}
D(\widetilde E)\leq 2D(E)\,, && \widetilde E\Delta B| \geq \frac{\delta(E)}6\,.
\end{align}
At this point \eqref{sharpest} follows at once by combining these two inequalities with the estimate \eqref{intro1} for the nearly spherical set $\widetilde E$. 

Note that in proving the reduction to the nearly spherical case we cannot use a regularity argument such as the one introduced by Cicalese and Leonardi in \cite{CL} to prove the stability of the isoperimetric inequality, see also \cite{AFM, BDV}. In fact no a priori regularity information can be hoped for the local minimizers of the functional $\F(E)$ whose Euler-Lagrange equation is not  even a differential equation.  Instead, the proof of \eqref{intro2} is obtained by a delicate combination of rearrangement  and mass transportation arguments and uses in a crucial way that $\alpha>1$. However, this is not just a technical assumption. Indeed there is a substantial difference between the case $\alpha>1$, which corresponds to a ``long-range'' interaction, and the ``short-range'' interaction case $\alpha\leq1$. Our impression is that the latter case will  require new ideas and a different approach.

\section{Reduction to a nearly spherical set\label{two}}

The goal of this section is to show that, in order to prove Theorem~\mref{Main}, one can reduce himself to a set whose boundary is the graph of a uniformly small function over the boundary of a unit ball. Such sets will be called nearly spherical sets, see Definition~\ref{def:nss}. More precisely, we will devote the section to show the next result.

\begin{prop}\label{mainreduction}
For every $\eps>0$ there exists a constant $K=K(N,\,\alpha,\,\eps)$ such that, for every $E\subseteq\R^N$ with $|E|=\omega_N$, either~(\ref{sharpest}) holds true for $E$, or there is an $\eps$-nearly spherical set $\widetilde E$ around $B$ satisfying
\begin{align}\label{standineq}
D(\widetilde E)\leq 2D(E)\,, && |\widetilde E\Delta B| \geq \frac{\delta(E)}6\,,
\end{align}
and such that the barycenter of $\widetilde E$ is at the origin.
\end{prop}

\subsection{Few facts about mass transportation}

In this paper we will use some very basic tools about mass transportation. Actually, all we need is only the definition of tranport map in a specific case, and a widely known existence property, and everything is contained in the next few lines. A reader who wish to know more about mass transportation can refer, for instance, to the book~\cite{V}.\par

Let $f,\,g:\R^N\to \R^+$ be two Borel functions such that $\int_{\R^N} f = \int_{\R^N} g <+\infty$. A \emph{transport map between $f$ and $g$} is any Borel function $\Phi:\R^N\to\R^N$ such that $\Phi_\# f = g$, that is, for every continuous, positive function $\varphi:\R^N\to\R^+$ one has
\[
\int_{\R^N} \varphi(z)g(z)\, dz = \int_{\R^N} \varphi(\Phi(y)) f(y)\,dy\,.
\]
If there exist a Borel function $\Phi^{-1}:\R^N\to\R^N$ such that $\Phi(\Phi^{-1}(z))=z$ for almost every $z$ such that $g(z)>0$, and $\Phi^{-1}(\Phi(y))=y$ for almost every $y$ such that $f(y)>0$, and the map $\Phi^{-1}$ is a transport map between $g$ and $f$, then we say that \emph{$\Phi$ is an invertible transport map between $f$ and $g$}.\par
In the particular case when $f$ and $g$ are two characteristic functions, that is, if $f=\Chi{H}$ and $g=\Chi{K}$ for two sets $H,\,K\subseteq\R^N$ of equal measure, a transport map $\Phi$ between $f$ and $g$ will also be called directly \emph{transport map between $H$ and $K$}. In this case, the above equality reads as
\begin{equation}\label{transportmap}
\int_K \varphi(z)\, dz = \int_H \varphi(\Phi(y)) \,dy\,.
\end{equation}
Notice that $\det(D\Phi)\equiv 1$ for every (regular enough) invertible transport map between two sets.\par
A useful property is that invertible transport maps always exist, in this setting. In other words, given any $f,\,g$ as above, there always exists at least an invertible transport map, see for instance~\cite[Theorem~6.2]{A}.

\subsection{Notations and preliminary estimates}

In this section we present few notations and a couple of simple but useful estimates. Here, as in the rest of the paper, $1<\alpha<N$ is a fixed constant. First of all, for every $x\in\R^N,\, r>0$, we denote by $B_x(r)$ the open ball with center in $x$ and radius $r$, and we set also $B_x=B_x(1)$, $B(r)=B_0(r)$, $B=B_0(1)$. We will also write, for any two Borel sets $G,\, H\subseteq\R^N$,
\begin{equation}\label{defIGH}
\I(G,H)= \int_G \int_H \frac 1{|y-x|^{N-\alpha}}\, dy\, dx\,,
\end{equation}
so that $\F(E)=\I(E,E)$. Moreover, for every $t>0$, we set
\begin{equation}\label{defpsi}
\psi(t) = \int_B \frac 1{|y-x|^{N-\alpha}}\,dx\,,
\end{equation}
where $y$ is any point such that $|y|=t$. Notice that $\psi:[0,+\infty)\to (0,+\infty)$ is a strictly decreasing, ${\rm C}^1$ function.\par

We now define the nearly spherical sets. Notice that this term has been used several times, with slightly different meanings. In particular, for our purposes we call nearly spherical sets those whose boundary is the graph of a function over the unit sphere, and this function is only required to be uniformly small. In other papers, the same function is required to be small in some stronger sense, for instance in ${\rm C}^1$.

\begin{defin}\label{def:nss}
A set $E_z\subseteq\R^N$ with $|E_z|=\omega_N$ is said an \emph{$\eps$-nearly spherical set around $B_z$}, for some $z\in\R^N$ and some $0<\eps<1$, if there exists a measurable function $u:\partial B \to (-\eps,\eps)$ such that
\begin{equation}\label{eq:defnss}
E_z= \big\{ z+(1+\rho)x:\, x\in \partial B,\, -1\leq \rho\leq u(x)\big\}\,.
\end{equation}
\end{defin}

We see now a simple consequence of Riesz inequality~(\ref{Riesz}).
\begin{lemma}\label{lemnew}
For any positive, measurable function $g:\R^N\to\R^+$, one has
\begin{equation}\label{rieszleft}
\int_{\R^N} \frac{g(y)}{|y|^{N-\alpha}} \,dy \leq \int_{\R^N} \frac{g^*(y)}{|y|^{N-\alpha}}\,dy\,.
\end{equation}
In particular, for any Borel set $H\subseteq\R^N$ and any point $x\in\R^N$, we have
\begin{equation}\label{eqstar}
\int_H \frac 1{|y-x|^{N-\alpha}}\,dy\leq \int_{B_x(r)} \frac 1{|y-x|^{N-\alpha}}\,dy\,,
\end{equation}
where $r=(|H|/\omega_N)^{1/N}$.
\end{lemma}
\begin{proof}
Let the positive, measurable function $g:\R^N\to\R^+$ be given. First of all we observe that, by the Monotone Convergence Theorem, to get~(\ref{rieszleft}) it is enough to consider the case when $g$ is bounded. Assume then that $g\leq C$, and let $\eps>0$ be a constant. Applying Riesz inequality~(\ref{Riesz}) with
\begin{align*}
f(z) = \frac 1{\omega_N\eps^N}\,\Chi{B(\eps)}(z)\,, && h(t)=\frac 1{t^{N-\alpha}}\,,
\end{align*}
we get
\begin{equation}\label{ucp}
\intmed_{B(\eps)} \int_{\R^N} \frac {g(y)}{|y-z|^{N-\alpha}}\, dy\,dz \leq
\intmed_{B(\eps)} \int_{B(r)} \frac {g^*(y)}{|y-z|^{N-\alpha}}\, dy\,dz\,.
\end{equation}
Let now $\delta>0$ be fixed, and notice that
\begin{equation}\label{stupest}\begin{split}
\intmed_{B(\eps)} \int_{B(\delta)} \frac {g(y)}{|y-z|^{N-\alpha}}\, dy\,dz &\leq C\, \intmed_{B(\eps)} \int_{B(\delta)} \frac 1{|y-z|^{N-\alpha}}\, dy\,dz \leq C \int_{B(\delta+\eps)} \frac 1{|y|^{N-\alpha}}\,dy \\
&\leq C'(\eps+\delta)^\alpha\,.
\end{split}\end{equation}
Since by the Dominated Convergence Theorem one has
\[
\lim_{\eps\to 0} \intmed_{B(\eps)} \int_{\R^N\setminus B(\delta)} \frac {g(y)}{|y-z|^{N-\alpha}}\, dy\,dz = \int_{\R^N\setminus B(\delta)} \frac{g(y)}{|y|^{N-\alpha}}\,dy\,,
\]
we deduce by~(\ref{stupest}) that
\[
\lim_{\eps\to 0} \intmed_{B(\eps)} \int_{\R^N} \frac {g(y)}{|y-z|^{N-\alpha}}\, dy\,dz = \int_{\R^N} \frac{g(y)}{|y|^{N-\alpha}}\,dy\,.
\]
Inserting this equality, and the corresponding one with $g$ replaced by $g^*$, into~(\ref{ucp}), we get~(\ref{rieszleft}).\par
Let now the Borel set $H\subseteq\R^N$ and the point $x\in\R^N$ be given. Applying~(\ref{rieszleft}) with $g(y)=\Chi H(x+y)$, we get $(y+x=Y$)
\[
\int_H \frac 1{|y-x|^{N-\alpha}} \,dy
=\int_{\R^N} \frac{g(y)}{|y|^{N-\alpha}} \,dy \leq \int_{\R^N} \frac{g^*(y)}{|y|^{N-\alpha}}\,dy
=\int_{B(r)} \frac 1{|y|^{N-\alpha}}\,dy
= \int_{B_x(r)} \frac 1{|y-x|^{N-\alpha}}\,dy\,,
\]
that is~(\ref{eqstar}).
\end{proof}

\begin{lemma}\label{generic}
There exists a continuous, increasing function $\tau_1:\R^+\to\R^+$, depending only on $N$ and on $\alpha$, such that $\tau_1(0)=0$ and for any two Borel sets $G,\, H\subseteq \R^N$ one has
\[
\I(G,H)\leq |G| \tau_1(|H|)\,.
\]
\end{lemma}
\begin{proof}
Let $x$ be any point of $G$, and let $r= (|H|/\omega_N)^{1/N}$. By~(\ref{eqstar}) one has
\[\begin{split}
\int_H \frac 1{|y-x|^{N-\alpha}}\, dy &\leq \int_{B_x(r)} \frac 1{|y-x|^{N-\alpha}}\, dy
= N\omega_N \int_0^r \frac 1{\rho^{N-\alpha}}\, \rho^{N-1}\,d\rho
=\frac{N\omega_N}\alpha\, r^\alpha\\
&= \frac{N\omega_N^{1-\frac\alpha N}}\alpha\, |H|^{\frac\alpha N}=:\tau_1(|H|)\,.
\end{split}\]
By integration over $x\in G$, we immediately get the thesis.
\end{proof}

\begin{lemma}\label{lessgeneric}
There exists a continuous, increasing function $\tau_2:\R^+\to\R^+$, depending only on $N$ and on $\alpha$, such that $\tau_2(0)=0$ and the following holds. For any three Borel sets $G,\, H,\,K\subseteq \R^N$ with $|H|=|K|$ and for any invertible transport map $\Phi$ between $H$ and $K$, one has
\begin{equation}\label{stima2}
\big|\I(G,H)-\I(G,K)\big| \leq \tau_2(|G|) \int_H 1\wedge |y-\Phi(y)|\, dy\,.
\end{equation}
\end{lemma}
\begin{proof}
Since $\Phi$ is an invertible transport map, by the symmetry of the problem we can assume without loss of generality that $\I(G,H)\geq \I(G,K)$. Indeed, by~(\ref{transportmap}) we have
\[
\int_K 1\wedge |z-\Phi^{-1}(z)|\, dz=\int_H 1\wedge |\Phi(y)-y|\, dy\,.
\]
Let us fix three points $x,\,y,\,z\in\R^N$. We start by establishing that
\begin{equation}\label{horrin}
\frac 1{|y-x|^{N-\alpha}}-\frac 1{|z-x|^{N-\alpha}} \leq (N-\alpha+1)\, \frac{1 \wedge |y-z|}{|y-x|^{N-\alpha+1}\wedge |y-x|^{N-\alpha}}\,.
\end{equation}
Indeed, if $|y-x|>|z-x|$, then the left hand side of the inequality is negative and the inequality is emptily true. Otherwise, the left hand side is smaller than
\[
(N-\alpha)\, \frac{|y-z|}{|y-x|^{N-\alpha+1}}\,.
\]
If $|y-z|\leq 1$, this gives~(\ref{horrin}). Otherwise, the left hand side of~(\ref{horrin}) is surely smaller than
\[
\frac 1{|y-x|^{N-\alpha}}=\frac{1 \wedge |y-z|}{|y-x|^{N-\alpha}}\leq \frac{1 \wedge |y-z|}{|y-x|^{N-\alpha+1}\wedge |y-x|^{N-\alpha}}
\,,
\]
thus~(\ref{horrin}) is shown.\par
Keeping in mind~(\ref{transportmap}) and the assumption $\I(G,H)\geq \I(G,K)$, by~(\ref{horrin}) we get
\[\begin{split}
|\I(G,H)-\I(G,K)| &= 
\int_G \int_H \frac 1{|y-x|^{N-\alpha}}\, dy\, dx-\int_G \int_K \frac 1{|z-x|^{N-\alpha}}\, dz\, dx\\
&=\int_G \int_H \frac 1{|y-x|^{N-\alpha}}- \frac 1{|\Phi(y)-x|^{N-\alpha}}\, dy\, dx\\
&\leq (N-\alpha+1) \int_G \int_H \frac{1 \wedge |y-\Phi(y)|}{|y-x|^{N-\alpha+1}\wedge |y-x|^{N-\alpha}}\, dy\,dx\\
&= (N-\alpha+1) \int_H 1 \wedge |y-\Phi(y)| \bigg(\int_G \frac 1{|y-x|^{N-\alpha+1}\wedge |y-x|^{N-\alpha}}\,dx\bigg) \, dy\,.
\end{split}\]
Let us now consider the integral in parentheses. Calling $r(G)= (|G|/\omega_N)^{1/N}$ the radius of the ball with the same volume as $|G|$, using the Riesz inequality as in Lemma~\ref{lemnew}, for every $y\in\R^N$ we have
\[
\int_G \frac 1{|y-x|^{N-\alpha+1}\wedge |y-x|^{N-\alpha}}\,dx\leq \int_{B_{r(G)}} \frac 1{|x|^{N-\alpha+1}\wedge |x|^{N-\alpha}}\, dx\,.
\]
Defining $\tau_2(|G|)/(N-\alpha+1)$ the latter integral, which is finite because $\alpha>1$, we conclude~(\ref{stima2}), so the proof is concluded.
\end{proof}

\subsection{Reduction to a small asymmetry}

This section is devoted to reduce ourselves to the case of sets with small asymmetry. In particular, we aim to prove the following continuity result, which is a non-quantitative version of Theorem~\mref{Main}.

\begin{lemma}\label{reduction}
For every $\mu>0$ there exists $\eta=\eta(\mu,\alpha,N)>0$ such that, for every set $E\subseteq\R^N$ with $|E|=\omega_N$ and $\delta(E)\geq \mu$, one has $D(E)\geq\eta$.
\end{lemma}

In order to prove this result, we start with the following rough estimate, which basically says that a very sparse set cannot have a small energy deficit.

\begin{lemma}\label{verybad}
There exists a constant $\xi=\xi(\alpha,N)>0$ such that every set $E\subseteq\R^N$ with $|E|=\omega_N$ and $\delta(E)\geq 2(\omega_N-\xi)$ satisfies
\begin{equation}\label{tvo}
D(E)>\frac{\omega_N^2}{5^N}\, \bigg(1-\frac 1{2^{N-\alpha}}\bigg)\,.
\end{equation}
\end{lemma}
\begin{proof}
We start observing that for every $x\in B(1/2)$ the inclusion $B_x(1/2)\subseteq B$ holds. We can then divide $B\times B$ as the disjoint union $\Gamma_1\cup\Gamma_2$, where every $(x,y)\in B\times B$ belongs to $\Gamma_1$ if $x\in B(1/2)$ and $y\in B_x(1/2)$, and to $\Gamma_2$ otherwise. Since $|y-x|\leq 1$ for every $(x,y)\in\Gamma_1$ and $|y-x|\leq 2$ for every $(x,y)\in\Gamma_2$, we immediately get the (not so precise) estimate
\begin{equation}\label{estiB}
\F(B)= \iint_{\Gamma_1} \frac 1{|y-x|^{N-\alpha}}+ \iint_{\Gamma_2} \frac 1{|y-x|^{N-\alpha}}
\geq |\Gamma_1| + \frac{|\Gamma_2|}{2^{N-\alpha}}
= \frac{\omega_N^2}{2^{N-\alpha}} +\frac{\omega_N^2}{4^N}\bigg(1-\frac 1{2^{N-\alpha}}\bigg)\,.
\end{equation}
Let us now consider a set $E\subseteq\R^N$ with $|E|=\omega_N$, and assume that $\delta(E)\geq 2(\omega_N-\xi)$ for a suitable $\xi$ to be specified later. For every ball $B_z$ of radius $1$ we have
\begin{equation}\label{gac}
|E\cap B_z| = \omega_N - |B_z\setminus E| = \omega_N - \frac{|B_z\Delta E|}2 \leq \omega_N - \frac{\delta(E)}2\leq \xi\,.
\end{equation}
Let $K=K(N)\in\N$ be a constant such that the annulus $B(2)\setminus B(1)$ can be covered with $K$ balls of radius $1$. Fix now any $x\in E$, and subdivide $E=E_1\cup E_2\cup E_3$, where $E_1=E\cap B_x(1)$, $E_2=E\cap (B_x(2)\setminus B_x(1))$, $E_3=E\setminus B_x(2)$. By~(\ref{gac}), $|E_1|\leq \xi$ and $|E_2|\leq K\xi$. By~(\ref{eqstar}),
\[
\int_{E_1} \frac 1{|y-x|^{N-\alpha}} \, dy \leq \int_{B_x((|E_1|/\omega_N)^{1/N})} \frac 1{|y-x|^{N-\alpha}} \, dy
=\frac{N\omega_N^{1-\frac\alpha N}}\alpha\, |E_1|^{\frac\alpha N}
\leq \frac{N\omega_N^{1-\frac\alpha N}}\alpha\, \xi^{\frac\alpha N}\,.
\]
Moreover, by construction
\begin{align*}
\int_{E_2} \frac 1{|y-x|^{N-\alpha}} \, dy \leq |E_2| \leq K\xi
\,, && \int_{E_3} \frac 1{|y-x|^{N-\alpha}} \, dy \leq \frac{|E_3|}{2^{N-\alpha}}
\leq \frac{\omega_N}{2^{N-\alpha}}\,.
\end{align*}
Putting together the above estimates, we get
\[
\int_E \frac 1{|y-x|^{N-\alpha}} \,dy\leq \frac{N\omega_N^{1-\frac\alpha N}}\alpha\, \xi^{\frac\alpha N}+K\xi+\frac{\omega_N}{2^{N-\alpha}}\,,
\]
and since this holds for a generic $x\in E$ we deduce
\[
\F(E) =\int_E \int_E \frac 1{|y-x|^{N-\alpha}} \,dy\,dx \leq \frac{N\omega_N^{2-\frac\alpha N}}\alpha\, \xi^{\frac\alpha N}+K\xi\omega_N+\frac{\omega_N^2}{2^{N-\alpha}}\,.
\]
Comparing this estimate with~(\ref{estiB}), we get
\[
D(E)= \F(B)-\F(E) \geq \frac{\omega_N^2}{4^N}\bigg(1-\frac 1{2^{N-\alpha}}\bigg)-\frac{N\omega_N^{2-\frac\alpha N}}\alpha\, \xi^{\frac\alpha N}-K\xi\omega_N\,,
\]
which proves the validity of~(\ref{tvo}) provided $\xi=\xi(N,\alpha)$ has been chosen small enough.
\end{proof}

We prove now a result concerning the energy of functions, instead of sets. More precisely, with a small abuse of notation, we extend~(\ref{defIGH}) and~(\ref{defF}) to $L^1$ functions $f,\,g:\R^N\to [0,1]$ as follows,
\begin{align*}
\I(f,g)= \int_{\R^N} \int_{\R^N} \frac {f(x)g(y)}{|y-x|^{N-\alpha}}\, dy\, dx\,, && \F(f) = \int_{\R^N} \int_{\R^N} \frac{f(x)f(y)}{|y-x|^{N-\alpha}}\, dy\, dx=\I(f,f)\,.
\end{align*}
Notice that $\F(\Chi{E})=\F(E)$ and $\I(\Chi{G},\Chi{H})=\I(G,H)$. The following estimates hold.
\begin{lemma}
For every $L^1$ function $g:\R^N\to [0,1]$, we have
\begin{align}\label{riesz}
\int_{\R^N} \frac{g(y)}{|y|^{N-\alpha}} \,dy \leq \int_{B(\sqrt[N]{\|g\|_{L^1}/\omega_N})} \frac 1{|y|^{N-\alpha}}\,dy\,, &&
\F(g) \leq \F\big(B\big(\sqrt[N]{\|g\|_{L^1}/\omega_N}\,\big)\big)\,,
\end{align}
and the right inequality is strict unless $g$ is the characteristic function of a ball.
\end{lemma}
\begin{proof}
To prove the left inequality it is enough to observe that, since $0\leq g\leq 1$, calling for brevity $r=\sqrt[N]{\|g\|_{L^1}/\omega_N}$ one has
\[\begin{split}
\int_{B(r)} \frac 1{|y|^{N-\alpha}}\,dy&-\int_{\R^N} \frac{g(y)}{|y|^{N-\alpha}} \,dy
=\int_{B(r)} \frac{1-g(y)}{|y|^{N-\alpha}}\,dy - \int_{\R^N\setminus B(r)} \frac{g(y)}{|y|^{N-\alpha}} \,dy\\
&\geq \int_{B(r)} \frac{1-g(y)}{r^{N-\alpha}}\,dy - \int_{\R^N\setminus B(r)} \frac{g(y)}{r^{N-\alpha}} \,dy
=\frac 1{r^{N-\alpha}}\,\bigg(\int_{B(r)} 1\,dy -\int_{\R^N} g(y)\,dy\bigg)\\
&=\frac 1{r^{N-\alpha}}\,\Big(\omega_N r^N -\|g\|_{L^1}\Big) =0\,.
\end{split}\]
Concerning the right inequality, for every function $\theta:\R^N\to [0,1]$ we denote by $\hat \theta:\R^N\to [0,1]$ the function given by $\hat \theta=\Chi{B(r)}$, being $r=\sqrt[N]{\|\theta\|_{L^1}/\omega_N}$. We claim that
\begin{align}\label{cicleo}
\I(f,\hat \theta)\geq \I(f,\theta) && \forall\, f,\,\theta:\R^N\to[0,1],\, f=f^*,\,\theta=\theta^*\,,
\end{align}
with equality only if $\hat \theta=\theta$ in the special case when $\hat f=f$ (the equality holds true only if $\hat\theta=\theta$ even without the assumption $\hat f=f$, but since we do not need this stronger fact will not prove it). Notice that this will readily imply the right inequality in~(\ref{riesz}), since of course $\widehat{g^*}=\hat g$, so applying~(\ref{Riesz}) once and then~(\ref{cicleo}) twice, calling again for brevity $r=\sqrt[N]{\|g\|_{L^1}/\Omega_N})$, we get
\[
\F(g)=\I(g,g)\leq \I(g^*,g^*)\leq \I(g^*,\hat g) \leq \I(\hat g,\hat g) =\F(\hat g) = \F(\Chi{B(r)})=\F(B(r))\,,
\]
which is the desired inequality. Concerning the equality cases, observe that $\F(g)=\F(B(r))$ if and only if $g^*=\hat g=\Chi{B(r)}$, and $\I(g,g)=\I(g^*,g^*)$. As noticed at the beginning, since $h(t)=t^{-(N-\alpha)}$ is strictly decreasing, the latter equality holds if and only if $g=g^*$ up to a translation. Summarizing, equality in the right inequality in~(\ref{riesz}) holds if and only if $g$ is the characteristic function of a ball. Thus, to conclude the proof we only have to establish~(\ref{cicleo}).\par
Let then $f,\,\theta:\R^N\to [0,1]$ be two functions such that $f^*=f$ and $\theta^*=\theta$. For every $\rho>0$, let us define
\[
\zeta(\rho) = \intmed_{\partial B(\rho)} \int_{R^N} \frac{f(z)}{|z-y|^{N-1}}\,dz \,d\H^{N-1}(y)\,.
\]
Let now $\rho_1<\rho_2$ be given, and for every $\eps>0$ let $g:\R^N\to \R^+$ be given by
\[
g= \Chi{B(\rho_1)} + \Chi{B(\rho_2+\eps)}-\Chi{B(\rho_2)}\,,
\]
so that
\begin{equation}\label{forforg}
g^*=\Chi{B(\rho_1+\delta)}=g+\big(\Chi{B(\rho_1+\delta)}-\Chi{B(\rho_1)}\big)-\big(\Chi{B(\rho_2+\eps)}-\Chi{B(\rho_2)}\big)\,,
\end{equation}
with $(\rho_1+\delta)^N = \rho_1^N + (\rho_2+\eps)^N-\rho_2^N$. Applying~(\ref{Riesz}) to $f$ and $g$, as usual with $h(t)=t^{-(N-\alpha)}$, keeping in mind that $f=f^*$ we obtain $\I(f,g)<\I(f,g^*)$. By sending $\eps$ to $0$, formula~(\ref{forforg}) then implies that $\zeta(\rho_1)\geq \zeta(\rho_2)$. That is, the function $\zeta$ is decreasing. In the special case when $f=\hat f$, that is, $f$ is the characteristic function of a ball, the fact that $\zeta$ is strictly decreasing is clear since so is the function $\psi$ defined in~(\ref{defpsi}).\par

To prove~(\ref{cicleo}), we can then call again for brevity $r=\sqrt[N]{\|\theta\|_{L^1}/\omega_N}$ and argue as in the first half of the proof. More precisely, recalling that $\theta=\theta^*$ and that $0\leq \theta\leq 1$, and calling then with a slight abuse of notation $\theta(\rho)=\theta(y)$ for any $|y|=\rho$, we have
\[\begin{split}
\I(f,\hat\theta)-\I(f,\theta)&=\int_0^{+\infty}\int_{\partial B(\rho)}\int_{\R^N}\frac{f(z)(\hat\theta(y)-\theta(y))}{|z-y|^{N-1}}\,dz\, d\H^{N-1}(y)\, d\rho\\
&=\int_0^{+\infty}\zeta(\rho)(\hat\theta(\rho)-\theta(\rho))\,N\omega_N\rho^{N-1} d\rho\\
&=\int_0^r\zeta(\rho)(1-\theta(\rho))\,N\omega_N\rho^{N-1} d\rho-\int_r^{+\infty}\zeta(\rho)\theta(\rho)\,N\omega_N\rho^{N-1} d\rho\\
&\geq\int_0^r\zeta(r)(1-\theta(\rho))\,N\omega_N\rho^{N-1} d\rho-\int_r^{+\infty}\zeta(r)\theta(\rho)\,N\omega_N\rho^{N-1} d\rho\\
&=\zeta(r)\Big(\|\hat\theta-\theta\|_{L^1(B(r))}-\|\theta\|_{L^1(\R^N\setminus B(r))}\Big)=0\,.
\end{split}\]
The inequality~(\ref{cicleo}) is then proved, and in the special case when $\hat f=f$, thus $\zeta$ is strictly decreasing, the inequality is strict unless $\zeta(\rho)=\zeta(r)$ for every $\rho$ such that $\hat\theta(\rho)\neq \theta(\rho)$, that is, unless $\hat\theta=\theta$.
\end{proof}
We have then the following result.
\begin{lemma}\label{lemven}
Let $f_n:\R^N\to [0,1]$ be a sequence of $L^1$ functions which weakly* converges in $L^\infty(\R^N)$ to $f:\R^N\to [0,1]$ with $\int f=\lim_{n\to\infty} \int f_n$. Then $\F(f)=\lim_{n\to\infty} \F(f_n)$.
\end{lemma}
\begin{proof}
Let $\eps>0$ be any given number, and let $\Omega\subseteq\R^N$ be a bounded open set such that $\int_{\R^N\setminus\Omega} f<\eps$. The assumption that $\int f=\lim_{n\to\infty} \int f_n$ implies also $\int_{\R^N\setminus\Omega} f_n<\eps$ for $n$ large enough.\par

For any function $g:\R^N\to [0,1]$, let us call $\hat g:\R^N\times \R^N\to [0,1]$ the function given by $\hat g(x,y)=g(x)g(y)$. Notice that the weak* convergence of $f_n$ to $f$ in $L^\infty(\R^N)$ implies also the weak* convergence of $\hat f_n$ to $\hat f$ in $L^\infty(\R^N\times\R^N)$. As a consequence, since the function $(x,y)\mapsto \Chi\Omega(x)\Chi\Omega(y)/|y-x|^{N-\alpha}$ belongs to $L^1(\R^N\times \R^N)$, we get
\begin{equation}\label{almall}
\F(f_n\Chi{\Omega}) \freccia{n\to\infty} \F(f\Chi{\Omega})\,.
\end{equation}
We observe now that, also by the left inequality in~(\ref{riesz}) ,
\[\begin{split}
|\F(f) - \F(f\Chi\Omega)|&=\I(f\Chi{\R^N\setminus\Omega},f\Chi{\R^N\setminus\Omega}) + 2\I(f\Chi\Omega,f\Chi{\R^N\setminus\Omega})
\leq 2\I(f,f\Chi{\R^N\setminus\Omega}) \\
&=2\int_{\R^N\setminus \Omega} f(x) \bigg(\int_{\R^N} \frac{f(y)}{|y-x|^{N-\alpha}}dy\bigg)\,dx\\
&\leq 2 \int_{B(\sqrt[N]{\|f\|_{L^1}/\omega_N})} \frac 1{|y|^{N-\alpha}}\, dy\int_{\R^N\setminus \Omega} f(x) \, dx \\
&\leq 2 \eps \int_{B(\sqrt[N]{\|f\|_{L^1}/\omega_N})} \frac 1{|y|^{N-\alpha}}\, dy\,,
\end{split}\]
and in the very same way
\[
|\F(f_n) - \F(f_n\Chi\Omega)| \leq 2 \eps \int_{B(\sqrt[N]{\|f_n\|_{L^1}/\omega_N})} \frac 1{|y|^{N-\alpha}}\, dy\,.
\]
Since $\|f_n\|_{L^1}\to \|f\|_{L^1}$ and $\eps$ is arbitrary, by~(\ref{almall}) we deduce the thesis.
\end{proof}

We can now give the proof of Lemma~\ref{reduction}.

\proofof{Lemma~\ref{reduction}}
Let $\{E_n\}$ be a sequence of sets in $\R^N$ such that $|E_n|=\omega_N$ for every $N$, and $D(E_n)\to 0$ for $n\to\infty$. To show the claim, we have to prove that necessarily $\delta(E_n)\to 0$.\par
The concentration-compactness Lemma by Lions (\cite{L}, see also~\cite{S}) ensures that, up to pass to a subsequence and to translate the sets, one of the three following possibilities hold:\\
\emph{vanishing:} for every $R>0$ one has $\lim_{n\to\infty} \sup_{x\in\R^N} |E_n \cap B_x(R)| = 0$.\\
\emph{compactness:} for every $\eps>0$ there exists $\overline R=\overline R(\eps)$ such that $\limsup_{n\to \infty} |E_n \setminus B(\overline R)| <\eps$.\\
\emph{dichotomy:} there exists $0<\lambda < \omega_N$ such that, for every $\eps>0$, there exist $\overline R=\overline R(\eps)$ and sets $E^1_n,\, E^2_n\subseteq E_n$ with $E^1_n\subseteq B(\overline R)$ such that
\begin{align*}
\limsup_{n\to\infty} \big||E^1_n| - \lambda\big|<\eps\,, &&
\limsup_{n\to\infty} \big||E^2_n| - (|\omega_N|-\lambda)\big|<\eps\,, &&
\lim_{n\to \infty} {\rm dist}(E^1_n,E^2_n) =\infty\,.
\end{align*}
We consider the three possibilities separately.\par

First of all, we can easily exclude the \emph{vanishing}. In fact, assume that the vanishing holds, and let $R=1$. For every $n\in\N$, we have
\[
\delta(E_n) = \inf_{x\in\R^N} |E_n \Delta B_x(1)|
= 2 \inf_{x\in\R^N} |B_x(1)\setminus E_n| = 2\omega_N - 2 \sup_{x\in\R^N} |B_x(1)\cap E_n|\,,
\]
which by definition of vanishing implies that $\delta(E_n)\to 2\omega_N$. By Lemma~\ref{verybad}, we find a contradiction with the assumption that $D(E_n)\to 0$, thus the vanishing is excluded.\par

We can now exclude also the \emph{dichotomy}. Indeed, let us assume that dichotomy holds, and let $0<\lambda<\omega_N$ and $E^1_n,\, E^2_n$ be as in the definition. Since $E^1_n$ and $E^2_n$ are disjoint for $n$ large enough (because their distance explodes), we have $E_n=E^1_n\cup E^2_n\cup E^3_n$, with $E^3_n = E_n \setminus (E^1_n\cup E^2_n)$. We fix now some positive $\eps$, to be specified in a moment, and call
\begin{align*}
R_1 = \bigg(\frac{\lambda+2\eps}{\omega_N}\bigg)^{1/N} && R_2 = \bigg(\frac{\omega_N-\lambda+2\eps}{\omega_N}\bigg)^{1/N}\
\end{align*}
the radii of two balls having volume $\lambda+2\eps$ and $\omega_N-\lambda+2\eps$ respectively. Since for $n$ big enough we have $|E^1_n|< \lambda+2\eps$ and $|E^2_n|<\omega_N-\lambda+2\eps$, and since balls maximize the energy among sets with the same volume, we immediately obtain the estimates
\begin{align*}
\F(E^1_n) &\leq \F(B(R_1))
= \bigg(\frac{\lambda+2\eps}{\omega_N}\bigg)^{1+\frac\alpha N} \F(B)\,, &&
\F(E^2_n) \leq \bigg(\frac{\omega_N-\lambda+2\eps}{\omega_N}\bigg)^{1+\frac\alpha N} \F(B)\,,
\end{align*}
which by strict convexity imply
\begin{equation}\label{la1}
\F(E^1_n)+\F(E^2_n) < \F(B) -\eps(9\tau_1(\omega_N)+1)
\end{equation}
as soon as $\eps$ has been chosen small enough. Keeping in mind that $|E^3_n|<3\eps$, again for $n$ large enough, by Lemma~\ref{generic} we have also
\begin{equation}\label{la2}
\F(E^3_n) + 2\I(E^3_n,E^1_n\cup E^2_n) \leq 9\eps \tau_1(\omega_N)\,.
\end{equation}
Putting together~(\ref{la1}) and~(\ref{la2}), and keeping in mind that the distance between $E^1_n$ and $E^2_n$ diverges, we can then evaluate the energy as
\[
\F(E_n) = \F(E^1_n) + \F(E^2_n) + \F(E^3_n) + 2\I(E^3_n,E^1_n\cup E^2_n) + 2\I(E^1_n,E^2_n)<\F(B)-\eps
\]
for every $n$ large enough. This clearly contradicts the fact that $D(E_n)\to 0$ for $n\to\infty$, hence also the dichotomy is excluded.\par

Summarizing, we have reduced ourselves to consider the case when \emph{compactness} holds. In this last case, let us call $f_n = \Chi{E_n}$. Up to a subsequence, $\{f_n\}$ weakly* converges in $L^\infty(\R^N)$ to some $L^1$ function $f:\R^N\to [0,1]$. The fact that \emph{compactness} holds readily implies that $\|f\|_{L^1}=\lim_{n\to\infty} \{f_n\}_{L^1}=\omega_N$. As a consequence, by Lemma~\ref{lemven} we have that $\F(f)=\lim_{n\to\infty} \F(f_n)=\lim_{n\to\infty} \F(E_n)=\F(B)$, where the last equality holds because $D(E_n)\to 0$. By the right inequality in~(\ref{riesz}), we deduce that $f=\Chi{B_z}$ for some $z\in\R^N$. By the weak* convergence of $f_n$ to $f$ we obtain
\[
\limsup_{n\to\infty} \delta(E_n) 
\leq \limsup_{n\to\infty} |E_n \Delta B_z|
=2 \limsup_{n\to\infty} |E_n\setminus B_z|
=2\limsup_{n\to\infty} \int_{\R^N} f_n(x)(1- f(x))\, dx=0
\]
as desired. This concludes the proof.
\end{proof}

\subsection{A set uniformly close to a ball}

Our next aim is to show that a set with small Fraenkel asimmetry can be slightly modified in order to be uniformly close to a ball. More precisely, we will show the following weaker version of Proposition~\ref{mainreduction}.

\begin{lemma}\label{firstreduction}
For every $0<\eps<1$ there exists some $\delta_\eps$ with the following property. For every $E\subseteq\R^N$ with $|E|=\omega_N$ and $\delta(E)<\delta_\eps$, there is another set $E'\subseteq\R^N$, still with $|E'|=\omega_N$, such that $B(1-\eps^2)\subseteq E' \subseteq B(1+\eps^2)$ and
\begin{align}\label{firstineq}
D(E')\leq D(E)\,, && \delta(E')=\delta(E)\,.
\end{align}
\end{lemma}
\begin{proof}
Let $E$ be a set as in the claim. Up to a translation, we can assume that $B$ is optimal for the Fraenkel asymmetry, that is, $\delta(E)=|E\Delta B|$. We construct the set $E'$ in two steps. First of all, we look for a set $E_1\subseteq B(1+\eps^2)$ with volume $\omega_N$ and such that
\begin{align}\label{firstineq0}
D(E_1)\leq D(E)\,, && (E_1\Delta E) \cap B(1+\eps^2/3)=\emptyset\,.
\end{align}
To do so, we define $G=E\setminus B(1+\eps^2)$. If $\delta_\eps$ is small enough, it is possible to define some $\widetilde G\subseteq B(1+\eps^2/2)\setminus (B(1+\eps^2/3)\cup E)$ such that $|\widetilde G|=|G|$. We can then set
\[
E_1 = (E \cup \widetilde G) \setminus G\,.
\]
By construction, $|E_1|=|E|$ and the right property in~(\ref{firstineq0}) holds, hence we have only to take care of the left inequality. We have
\begin{equation}\label{tainan}\begin{split}
\F(E_1)&= \I(E_1,E_1)
=\I(E,E) + \I(E,\widetilde G)-\I(E,G)+\I(E_1,\widetilde G)-\I(E_1,G)\\
&\geq \F(E)+2\I(E,\widetilde G)-2\I(E,G) -2\I(G,\widetilde G)\\
&\geq \F(E) +2\I(B,\widetilde G) -2\I(B,G)-2\I(B\setminus E,\widetilde G)-2\I(E\setminus B,G)-2\I(G,\widetilde G)\,.
\end{split}\end{equation}
By Lemma~\ref{generic}, since $|\widetilde G|=|G|\leq |E\setminus B|=|B\setminus E|=\delta(E)/2 \leq \delta_\eps$, we can estimate
\[
2\I(B\setminus E,\widetilde G)+2\I(E\setminus B,G)+2\I(G,\widetilde G)
\leq 6\tau_1(\delta_\eps) |G|\,.
\]
On the other hand, by construction and by definition of $\psi$ we have
\begin{align*}
\I(B,\widetilde G) = \int_{\widetilde G} \psi(|x|)\,dx \geq |G| \psi(1+\eps^2/2)\,, &&
\I(B,G) = \int_G \psi(|x|)\,dx \leq |G| \psi(1+\eps^2)\,.
\end{align*}
Inserting the last estimates into~(\ref{tainan}), we get
\[
\F(E_1)-\F(E) \geq 2 |G| \big( \psi(1+\eps^2/2)-\psi(1+\eps^2) - 3\tau_1(\delta_\eps)\big)\,,
\]
and since $\tau_1$ is continuous and increasing, with $\tau_1(0)=0$, as soon as $\delta_\eps\ll \eps^2$ we get $\F(E_1)\geq \F(E)$, so also the left inequality in~(\ref{firstineq0}) is obtained and then~(\ref{firstineq0}) is established.\par

We now repeat the same procedure to find a set $E'\supseteq B(1-\eps^2)$ satisfying $|E'|=\omega_N$ and
\begin{align}\label{firstineq1}
D(E')\leq D(E_1)\,, && E_1\Delta E'\subseteq B(1-\eps^2/3)\,.
\end{align}
More precisely, we define $H=B(1-\eps^2)\setminus E_1$, we let $\widetilde H\subseteq (B(1-\eps^2/3)\cap E_1)\setminus B(1-\eps^2/2)$ be a set with $|\widetilde H|=|H|$, and we set $E'=(E_1\cup H)\setminus \widetilde H$. By construction we have that $|E'|=\omega_N$, that $B(1-\eps^2)\subseteq E'\subseteq B(1+\eps^2)$, and that the right inclusion in~(\ref{firstineq1}) holds. In addition, the very same calculation as in~(\ref{tainan}) now gives
\[
\F(E')\geq \F(E_1) +2\I(B,H)-2\I(B,\widetilde H) -2\I(E_1\setminus B,\widetilde H)-2\I(B\setminus E_1,H)-2\I(H,\widetilde H)\,,
\]
and as before by Lemma~\ref{generic} we can estimate
\[
2\I(E_1\setminus B,\widetilde H)+2\I(B\setminus E_1,H)+2\I(H,\widetilde H) \leq 6\tau_1(\delta_\eps) |H|\,,
\]
as well as
\begin{align*}
\I(B,H) = \int_H \psi(|x|)\,dx \geq |H| \psi(1-\eps^2)\,, &&
\I(B,\widetilde H) = \int_{\widetilde H} \psi(|x|)\,dx \leq |H| \psi(1-\eps^2/2)\,.
\end{align*}
Therefore,
\[
\F(E')-\F(E_1)\geq 2|H|\big( \psi(1-\eps^2)-\psi(1-\eps^2/2) - 3\tau_1(\delta_\eps)\big)\geq 0\,,
\]
where the last inequality holds true as soon as $\delta_\eps$ has been chosen small enough. Thus, also the left inequality in~(\ref{firstineq1}) is established, so~(\ref{firstineq1}) is proved.\par
Summarizing, we have defined a set $B(1-\eps^2)\subseteq E'\subseteq B(1+\eps^2)$ with $|E'|=\omega_N$. The left inequality in~(\ref{firstineq}) follows by the left inequalities in~(\ref{firstineq0}) and~(\ref{firstineq1}), thus to conclude the proof we only have to show the right equality in~(\ref{firstineq}).\par

Notice that the right properties in~(\ref{firstineq0}) and~(\ref{firstineq1}) imply that
\[
E'\Delta E\subseteq B(1-\eps^2/3) \cup \Big(\R^N \setminus B(1+\eps^2/3)\Big)\,.
\]
As a consequence, for every $x\in\R^N$ with $|x|\leq \eps^2/3$ we have $|E'\Delta B_x|=|E\Delta B_x|$, so
\begin{equation}\label{fintan1}
|E'\Delta B_x|=|E\Delta B_x|\geq \delta(E)\,.
\end{equation}
On the other hand, take any $x\in\R^N$ with $|x|>\eps^2/3$. Keeping in mind that
\[
B\Delta B_x \subseteq B \Delta E \cup E\Delta E' \cup E'\Delta B_x \,,
\]
we have by construction
\begin{equation}\label{fintan2}
|E'\Delta B_x| \geq |B\Delta B_x| - |B\Delta E|- |E\Delta E'| \geq |B\Delta B_x| - 3\delta(E)\geq \delta(E)\,,
\end{equation}
where the last inequality holds true as soon as $\delta_\eps$ is small enough with respect to $\eps$, since $\delta(E)\leq \delta_\eps$ while $|B\Delta B_x|$ can be bounded from below with a strictly positive constant depending only on $N$ and $\eps$. Since for every $x\in\R^N$ we have the validity either of~(\ref{fintan1}) or of~(\ref{fintan2}), we deduce that $\delta(E')=|E'\Delta B|=\delta(E)$, thus the right equality in~(\ref{firstineq}).
\end{proof}

\subsection{The nearly spherical set around a ball}

In this section we show that any set uniformly close to a ball can be reduced to a nearly spherical set.

\begin{prop}\label{defex}
There exists $0<\eps_1\ll 1$ depending on $N$ and on $\alpha$ with the following property. Let $0<\eps<\eps_1$, let $E'\subseteq\R^N$ be a set of volume $\omega_N$, let $z\in\R^N$, and assume that $B_z(1-\eps)\subseteq E' \subseteq B_z(1+\eps)$. Then, either the estimate~(\ref{sharpest}) holds true for $E'$ with a suitable $C$ depending only on $N$ and $\alpha$, or there exists a set $E_z$, which is $\eps$-nearly spherical around $B_z$ and satisfies
\begin{align}\label{secondineq}
D(E_z)\leq 2D(E')\,, && |E_z\Delta B_z|\geq \frac{\delta(E')}6\,.
\end{align}
\end{prop}

To show the proposition, we need a preparatory lemma.

\begin{lemma}\label{prelem}
There exist constants $\eps_1$ and $C$ only depending on $N$ and on $\alpha$ such that, for any $\eps,\,E'$ and $z$ as in Proposition~\ref{defex} the following holds. If~(\ref{sharpest}) does not hold true for $E'$, then there exist two functions $u^\pm:\S^{N-1}\to [0,\eps)$ so that, defining
\begin{equation}\label{defE''}
E'' = \big\{ z + tx: \, t \in [0,1-u^-(x))\cup (1,1+u^+(x)),\ x\in \S^{N-1}\big\}\,,
\end{equation}
the set $E''$ has volume $\omega_N$ and satisfies
\begin{align}\label{preparineq}
D(E'')\leq D(E')\,, && \delta(E'')\geq \frac{\delta(E')}2\,.
\end{align}
\end{lemma}
\begin{proof}
Let us assume, just for simplicity of notation, that $z=0$. We can immediately define $u^\pm:\S^{N-1}\to\R^+$ as the two functions such that, for every $x\in \S^{N-1}$, we have
\begin{align*}
\int_1^{1+u^+(x)} t^{N-1}\, dt = \int_1^{+\infty} t^{N-1} \Chi{E'}(tx)\,dt\,, &&
\int_{1-u^-(x)}^1 t^{N-1}\, dt = \int_0^1 t^{N-1} \Chi{^c\!E'}(tx)\,dt\,.
\end{align*}
Defining then $E''$ according to~(\ref{defE''}), we call now
\begin{align*}
G^+=(E'\setminus E'')\setminus B\,, && \widetilde G^+ = (E''\setminus E')\setminus B\,, \\
G^-=B\cap (E''\setminus E')\,, && \widetilde G^- = B\cap (E'\setminus E'')\,.
\end{align*}
Notice that $\widetilde G^+\subseteq B(1+\eps)\setminus B$ and $\widetilde G^-\subseteq B\setminus B(1-\eps)$, and moreover $|E''|=\omega_N$, in particular $|\widetilde G^+|=|G^+|$, $|\widetilde G^-|=|G^-|$.\par

We write now $H=\widetilde G^+\cup G^-$ and $K=G^+\cup \widetilde G^-$, and we define the function $\Phi:H\to K$ as follows. For any $y\in H$, we let $\Phi(y)=\varphi(y) \frac y{|y|}$ where, if $|y|\geq 1$,
\begin{equation}\label{defphi1}
\int_1^{|y|} \Chi{E''\setminus E'}\bigg(t\,\frac y{|y|}\bigg) t^{N-1}\,dt = \int_1^{\varphi(y)} \Chi{E'\setminus E''}\bigg(t\,\frac y{|y|}\bigg)t^{N-1} \,dt\,,
\end{equation}
and similarly, if $|y|\leq 1$,
\[
\int_{\varphi(y)}^1 \Chi{E'\setminus E''}\bigg(t\,\frac y{|y|}\bigg) t^{N-1}\,dt=\int_{|y|}^1 \Chi{E''\setminus E'}\bigg(t\,\frac y{|y|}\bigg)t^{N-1}\,dt\,.
\]
It is simple to notice that $\Phi$ is an invertible transport map between $H$ and $K$ (in fact, it is a sort of ``radial version'' of the well-known Knothe map). As a consequence, we can apply Lemma~\ref{lessgeneric} four times, so~(\ref{stima2}) implies
\begin{equation}\label{dawn}\begin{split}
\I(K,E'\setminus B)&-\I(H,E'\setminus B)
+\I(H,B\setminus E')-\I(K,B\setminus E')+\I(K,E''\setminus B)\\
&\phantom{-\I(K,E'\setminus B)\ }-\I(H,E''\setminus B)+\I(H,B\setminus E'')-\I(K,B\setminus E'')\\
&\leq 2\Big(\tau_2(|E'\setminus B|)+\tau_2(|B\setminus E'|)\Big) \int_H |y-\Phi(y)|\, dy \,,
\end{split}\end{equation}
also keeping in mind that $|E'\setminus B|=|E''\setminus B|$ and $|B\setminus E'|=|B\setminus E''|$. Moreover, since $B(1-\eps)\subseteq E'\subseteq B(1+\eps)$ and the same is true for $E''$, and since by construction for every $y\in H$ one has $y/|y| = \Phi(y)/|\Phi(y)|$ and $|\Phi(y)|\geq |y|$, then for every $y\in H$
\[
\psi(|y|)-\psi(|\Phi(y)|) \geq c |y-\Phi(y)|\,,
\]
where $c=\min \{|\psi'(t)|:\, 1-\eps\leq t\leq 1+\eps\}$. By integration, we get
\[
\I(H,B)-\I(K,B) \geq c \int_H |y-\Phi(y)|\, dy\,.
\]
Using this inequality together with~(\ref{dawn}), we can now evaluate
\begin{equation}\label{airtrain}\begin{split}
\F(&E'')-\F(E') = \I(H,E')-\I(K,E')+\I(H,E'')-\I(K,E'')\\
&=2\I(H,B) -2\I(K,B)+\I(H,E'\setminus B)-\I(K,E'\setminus B)+\I(K,B\setminus E')\\
&\hspace{20pt}-\I(H,B\setminus E')+\I(H,E''\setminus B)-\I(K,E''\setminus B)+\I(K,B\setminus E'')-\I(H,B\setminus E'')\\
&\geq 2\Big(c -\tau_2(|E'\setminus B|)-\tau_2(|B\setminus E'|)\Big) \int_H |y-\Phi(y)|\, dy \geq 0\,,
\end{split}\end{equation}
where the last inequality holds true if $\eps_1$ has been chosen small enough. Indeed, keeping in mind that $B(1-\eps)\subseteq E'\subseteq B(1+\eps)$ and that $\tau_2(t)\searrow 0$ for $t\searrow 0$, we get that $\tau_2(|E'\setminus B|)+\tau_2(|B\setminus E'|)$ is arbitrarily small if $\eps\ll 1$. The constant $c$, instead, converges to $-\psi'(1)>0$ for $\eps\ll 1$.\par

Summarizing, we have found a set $E''$, defined through~(\ref{defE''}), such that $|E''|=\omega_N$ and $D(E'')\leq D(E')$. To conclude the proof, then, we have to show that either $\delta(E'')\geq \delta(E')/2$, so that~(\ref{preparineq}) holds true, or~(\ref{sharpest}) is satisfied by $E'$ with a suitable constant $C=C(N,\alpha)$.\par

Let us call $B''$ a suitable translation of $B$ such that $\delta(E'')=|E''\Delta B''|$. If $|E''\Delta B''| > \delta(E')/2$ we are done, so we assume that the opposite inequality holds true. By definition of Fraenkel asymmetry of $E'$, we have
\[
\delta(E') \leq |E'\Delta B''| \leq |E'\Delta E''| + |E''\Delta B''|\leq |E'\Delta E''| + \frac{\delta(E')}2\,,
\]
which gives, by construction,
\begin{equation}\label{bycon}
|H| = \frac{|E'\Delta E''|}2 \geq \frac{\delta(E')}4\,.
\end{equation}
For any direction $\nu\in\S^{N-1}$, let us now call $G_\nu = H\cap \nu\R$ the section of $H$ in direction $\nu$, and we subdivide $G_\nu=G^{\rm ext}_\nu\cup G^{\rm int}_\nu$, where $G^{\rm ext}_\nu=G_\nu\setminus B$ and $G^{\rm int}_\nu=G_\nu\cap B$. Notice that $(E''\setminus B)\cap \nu\R$ is a segment, in particular it is the segment $(1,1+u^+(x))\nu$. Every point of $\widetilde G^+\cap \nu\R$ is in this segment, while every point of $G^+\cap \nu\R$ is outside it. As a consequence, for any $y$ in a subset of $G^{\rm ext}_\nu$ of length $\H^1(G^{\rm ext}_\nu)/2$ one has $|\Phi(y)-y|>\H^1(G^{\rm ext}_\nu)/2$, thus
\[
\int_{G^{\rm ext}_\nu} |\Phi(y)-y|\, d\H^1 \geq \frac{\big(\H^1(G^{\rm ext}_\nu)\big)^2}4\,.
\]
Similarly, for any $y$ in a subset of $G^{\rm int}_\nu$ of length $\H^1(G^{\rm int}_\nu)/2$ it is $|\Phi(y)-y|>\H^1(G^{\rm int}_\nu)/2$, thus
\[
\int_{G^{\rm int}_\nu} |\Phi(y)-y| \, d\H^1 \geq \frac{\big(\H^1(G^{\rm int}_\nu)\big)^2}4\,.
\]
The last two inequalities imply
\[
\int_{G_\nu} |\Phi(y)-y|\, d\H^1\geq \frac{\big(\H^1(G_\nu)\big)^2}8\,,
\]
so an integration over $\S^{N-1}$ together with~(\ref{bycon}) gives
\[\begin{split}
\frac{\delta(E')}4 &\leq |H|
\leq (1+\eps)^{N-1} \int_{\S^{N-1}} \H^1(G_\nu)\,d\nu
\leq (1+\eps)^{N-1} \sqrt{N\omega_N} \sqrt{\int_{\S^{N-1}} \H^1(G_\nu)^2\,d\nu}\\
&\leq (1+\eps)^{N-1} \sqrt{8N\omega_N} \sqrt{\int_{\S^{N-1}} \int_{G_\nu} |\Phi(y)-y|\,d\H^1\,d\nu}\\
&\leq \frac{(1+\eps)^{N-1}}{(1-\eps)^{\frac{N-1}2}} \,\sqrt{8N\omega_N} \sqrt{\int_H |\Phi(y)-y|\,dy}\,.
\end{split}\]
Keeping in mind~(\ref{airtrain}), for $\eps$ small enough we deduce
\[
D(E') = \F(B)-\F(E') \geq \F(E'')-\F(E') \geq \frac{-\psi'(1)}{129N\omega_N} \, \delta(E')^2\,,
\]
which is exactly~(\ref{sharpest}), as desired. This concludes the proof.
\end{proof}

We are now in position to give the proof of Proposition~\ref{defex}.

\proofof{Proposition~\ref{defex}}
Let $E'$ be a set satisfying the assumptions. If the estimate~(\ref{sharpest}) holds true, there is nothing to prove. Otherwise, by Lemma~\ref{prelem} we have two functions $u^\pm:\S^{N-1}\to [0,\eps)$ such that the set $E''$ defined by~(\ref{defE''}) satisfies the inequalities~(\ref{preparineq}).\par
We start by replacing $u^\pm$ with two new functions $\tilde u^\pm:\S^{N-1}\to [0,\eps)$ which are locally constant. More precisely, we claim the existence of two functions $\tilde u^\pm:\S^{N-1}\to [0,\eps)$ such that the following holds. First of all, $\S^{N-1}$ is the piecewise disjoint union of finitely many sets $U_i$, so that $\tilde u^+\equiv u^+_i$ and $\tilde u^-\equiv u^-_i$ on each $U_i$, and that
\begin{equation}\label{smallareas}
{\rm diam}(U_i) \leq \min\{u^+_i,\, u^-_i\}\qquad \forall\, i:\, \min\{u^+_i,\, u^-_i\} >0\,.
\end{equation}
In addition, the set $\widetilde E''$ defined as in~(\ref{defE''}) with $u^\pm$ replaced by $\tilde u^\pm$ satisfies the following slightly weaker version of~(\ref{preparineq}),
\begin{align}\label{slightlyweaker}
D(\widetilde E'')\leq 2 D(E')\,, && \delta(\widetilde E'')\geq \frac{\delta(E')}3\,.
\end{align}
The validity of the claim is obvious. Indeed, $u^\pm$ can be written as strong limits of functions as $\tilde u^\pm$, and the corresponding sets converge to $E''$ both in terms of the energy deficit $D(\cdot)$ and of the Fraenkel asymmetry $\delta(\cdot)$.\par

To define the nearly spherical set $E_z$, we need to give a function $u:\S^{N-1}\to (-\eps,\eps)$ according to Definition~\ref{def:nss}. We will define $u$ separately on each set $U_i$. Suppose first that $\min\{u^+_i,\, u^-_i\} =0$: in this case, we simply set $u=u^+_i$ if $u^-_i=0$, and $u=-u^-_i$ if $u^+_i=0$ (hence, $u=0$ if $u^+_i=u^-_i=0$).\par

Let us now assume that $\min\{u^+_i,\, u^-_i\} >0$. Let us then write $U_i$ as the disjoint union of two sets $L_i$ and $R_i$, where
\begin{equation}\label{wisechoice}
\H^{N-1}(L_i) \Big(1-(1-u^-_i)^N \Big)= \H^{N-1}(R_i) \Big((1+u^+_i)^N-1\Big)\,,
\end{equation}
and in the set $U_i=L_i\cup R_i$ we define then $u$ as
\[
u =\Chi{L_i} u^+_i - \Chi{R_i}u^-_i\,.
\]
At this stage, we have completely defined the function $u:\S^{N-1}\to (-\eps,\eps)$, thus also the set $E_z$ is determined according to~(\ref{eq:defnss}). What we have to do to complete the proof, is to prove the validity of~(\ref{secondineq}).\par
Let us start defining the sets
\begin{align*}
F= E_z \setminus B_z\,, && D = B_z\setminus E_z\,, && \widetilde F= \widetilde E''\setminus B_z\,, && \widetilde D = B_z\setminus \widetilde E''\,.
\end{align*}
Moreover, calling $K_i$ the cones $K_i=U_i\times \R^+$ we also define the intersections
\begin{align*}
F_i=F\cap K_i\,, && D_i =D\cap K_i \,, && \widetilde F_i=F\cap K_i\,, && \widetilde D_i =\widetilde D\cap K_i\,.
\end{align*}
By construction, since $-\tilde u^- \leq u \leq \tilde u^+$, we have the inclusions $F\subseteq \widetilde F$ and $D\subseteq \widetilde D$. On the other hand, for each $i$ we have that $|F_i|+|D_i| \geq (|\widetilde F_i|+|\widetilde D_i|)/2$, hence summing over $i$ we get $|E_z\Delta B_z|\geq |\widetilde E''\Delta B_z|/2$. By the right estimate in~(\ref{slightlyweaker}), we get then
\[
|E_z\Delta B_z|\geq \frac{|\widetilde E''\Delta B_z|}2
\geq \frac{\delta(\widetilde E'')}2\geq \frac{\delta(E')}6\,,
\]
so the right estimate in~(\ref{secondineq}) is obtained, and we only need to get the left one.\par
Thanks to~(\ref{wisechoice}), for every $i$ we have $|\widetilde F_i\setminus F_i|=|\widetilde D_i\setminus D_i|$. We want then to define an invertible transport map $\Phi$ between $\widetilde F\setminus F$ and $\widetilde D\setminus D$, in such a way that the restriction of $\Phi$ to every $\widetilde F_i\setminus F_i$ is an invertible transport map on $\widetilde D_i\setminus D$. For every $i$ such that $\min\{u^+_i,\, u^-_i\}=0$ this is emptily done, since $\widetilde F_i\setminus F_i=\widetilde D_i\setminus D=\emptyset$.\par

Consider then an index $i$ such that $\min\{u^+_i,\, u^-_i\}>0$. In this case, we can take any invertible transport map $\tau_i:\S^{N-1}\to\S^{N-1}$ between $\Chi{R_i}$ and $\frac{1-(1-u^-_i)^N}{(1+u^+_i)^N-1}\,\Chi{L_i}$; notice that $\tau_i$ is a map between $R_i$ and $L_i$. Observe that
\begin{align*}
\widetilde F_i &= (L_i\cup R_i) \times (1,1+u^+_i)\,, &
F_i &=L_i \times (1,1+u^+_i)\,,\\ 
\widetilde D_i&= (L_i\cup R_i) \times (1-u^-_i,1)\,, &
D_i &= R_i \times (1-u^-_i,1)\,.
\end{align*}
As a consequence, we can define the function $\Phi$ between $\widetilde F_i\setminus F_i$ and $\widetilde D_i\setminus D_i$ simply as
\[
\Phi(t\nu) = g_i(t) \tau_i(\nu)\qquad \forall\, \nu\in R_i,\, t\in (1,1+u^+_i)\,,
\]
where
\[
\H^{N-1}(R_i)(t^N-1)=\H^{N-1}(L_i)\big(g_i(t)^N-(1-u^-_i)^N\big) \,.
\]
By construction, $\Phi:\widetilde F\setminus F\to \widetilde D\setminus D$ is clearly an invertible transport map and, keeping in mind~(\ref{smallareas}) and the definition of $g_i$, we also have
\[
|y|-|\Phi(y)| \geq \frac 12\, |y-\Phi(y)|\,,
\]
since $\eps$ is small. As a consequence, calling $c=\min\{-\psi'(t):\, 1-\eps<t<1+\eps\}>0$, we get
\begin{equation}\label{goodest}\begin{split}
\I(B,\widetilde D\setminus D)-\I(B,\widetilde F\setminus F)
&=\int_{\widetilde F\setminus F} \psi(|\Phi(y)|)-\psi(|y|)\,dy
\geq \int_{\widetilde F\setminus F} c\big(|y|-|\Phi(y)|\big)\,dy\\
&\geq \frac c2 \int_{\widetilde F\setminus F} |y-\Phi(y)|\,dy\,.
\end{split}\end{equation}
We can now write
\begin{equation}\label{subdiv}\begin{split}
\F(E_z)-\F(\widetilde E'')&=
\I(B\setminus D\cup F,B\setminus D\cup F)-\I(B\setminus \widetilde D\cup \widetilde F,B\setminus \widetilde D\cup \widetilde F)\\
&=2\I(B,\widetilde D\setminus D)-2\I(B,\widetilde F\setminus F) +C_1+C_2+C_3+C_4\,,
\end{split}\end{equation}
where
\begin{align*}
C_1=\I(\widetilde D,\widetilde F\setminus F)-\I(\widetilde D,\widetilde D\setminus D)\,, &&
C_2=\I(D,\widetilde F\setminus F)-\I(D,\widetilde D\setminus D)\,,\\
C_3=-\I(\widetilde F,\widetilde F\setminus F)+\I(\widetilde F,\widetilde D\setminus D)\,, &&
C_4=-\I(F,\widetilde F\setminus F)+\I(F,\widetilde D\setminus D)\,.
\end{align*}
Applying Lemma~\ref{lessgeneric} four times, each time with $H=\widetilde F\setminus F$, and with $G$ equal to $\widetilde D$, $D$, $\widetilde F$ and $F$ respectively, we get
\[\begin{split}
|C_1|+|C_2|+|C_3|+|C_4| &\leq \Big(\tau_2(|\widetilde D|)+\tau_2(|D|)+\tau_2(|\widetilde F|)+\tau_2(|F|)\Big)\int_{\widetilde F\setminus F} |y-\Phi(y)|\,dy\\
&\leq 4\tau_2 \big((1+\eps)^N-(1-\eps)^N\big)\int_{\widetilde F\setminus F} |y-\Phi(y)|\,dy\,.
\end{split}\]
This estimate, together with~(\ref{subdiv}) and~(\ref{goodest}), implies that $\F(E_z)\geq \F(\widetilde E'')$, or equivalentely $D(E_z)\leq D(\widetilde E'')$, if $\eps$ has been chosen small enough. Keeping in mind the left estimate in~(\ref{slightlyweaker}), also the left estimate in~(\ref{secondineq}) follows, hence the proof is concluded.
\end{proof}

\begin{remark}\label{remcont}
We point out that the sets $E_z$ given by Proposition~\ref{defex} can be chosen in such a way that the barycenter $Bar(z)$ of $E_z$ is a continuous function of $z$. To show that, we remind that Lemma~\ref{prelem} and Proposition~\ref{defex} can be applied to any $z\in\R^N$ such that $B_z(1-\eps)\subseteq E'\subseteq B_z(1+\eps)$. Let us call for a moment $u^\pm_z$ and $\tilde u^\pm_z$ the functions $u^\pm$ and $\tilde u^\pm$ used in the proofs of the two results, to highlight their dependance on the parameter $z$. A quick look to the proof of Lemma~\ref{prelem} ensures that the functions $u^\pm_z$ depend continuously on $z$ in $L^1(\S^{N-1})$. Analogously, a quick look to the proof of Proposition~\ref{defex} ensures that the functions $\tilde u^\pm_z$ can be constructed to depend continuously on $z$ in $L^1(\S^{N-1})$. As a simple consequence, we get that the functions $\Chi{E_z}$ depend continuously on $z$ in $L^1(\S^{N-1})$. This clearly implies the claim.
\end{remark}

\begin{lemma}\label{direction}
Let $\eps<\eps_1$, and let $E'\subseteq\R^N$ be a set such that $|E'|=\omega_N$ and $B(1-\eps^2)\subseteq E' \subseteq B(1+\eps^2)$. Then, for every $z$ such that $|z|=\eps/2$, the set $E_z$ given by Proposition~\ref{defex} satisfies $(Bar(z)-z)\cdot z<0$, where $Bar(z)$ denotes the barycenter of $E_z$.
\end{lemma}
\begin{proof}
Let us fix a vector $z$ with $|z|=\eps/2$, and let $E_z$ be the set given by Proposition~\ref{defex}, which can be applied since by construction
\[
B(1-\eps)\subseteq B(1-\eps^2-|z|) \subseteq E'\subseteq B(1+\eps^2+|z|)\subseteq B(1+\eps)\,.
\]
Let us call $H^+=\{x\in\R^N:\, x\cdot z>|z|^2\}$ and $H^-=\{x\in\R^N:\, x\cdot z<|z|^2\}$. Since $B(1-\eps^2)\subseteq E' \subseteq B(1+\eps^2)$, then of course
\begin{align*}
E'\cap H^+ \subseteq B(1+\eps^2)\cap H^+\,, && E'\cap H^- \supseteq B(1-\eps^2)\cap H^-\,.
\end{align*}
By the proofs of Lemma~\ref{prelem} and of Proposition~\ref{defex}, it is clear that the same inclusions hold also with $E_z$ in place of $E'$. Therefore, calling $Z=\big(B(1+\eps^2)\cap H^+ \big)\cup \big( B(1-\eps^2)\cap H^- \big)$, and denoting by $Bar(F)$ the barycenter of any set $F$ (so in particular $Bar(z)=Bar(E_z)$), we have
\[
Bar(E_z-z)\cdot z < Bar(Z-z) \cdot z<0\,,
\]
where the last inequality holds because $\eps<\eps_1$ is small. This concludes the thesis.
\end{proof}

\subsection{Adjustment of the barycenter}

The aim of this section is to prove Proposition~\ref{mainreduction}. Notice that each set $E_z$ defined in Proposition~\ref{defex} is already a nearly spherical set satisfying~(\ref{standineq}), so we only have to take care of the barycenter. In other words, all we have to do is to find a suitable $z\in\R^N$ such that the set $E_z$ has barycenter precisely at $z$. To do so, we will make use of the following well-known property, of which we give a proof just for completeness.

\begin{lemma}\label{ramiro}
Let $F:\overline B\to \overline B$ be a continuous function whose restriction $f$ to $\S^{N-1}=\partial B$ is mapped on $\S^{N-1}$ with $f(x)\neq -x$ for every $x\in \S^{N-1}$. Then, there exists some $x\in B$ such that $F(x)=0$.
\end{lemma}
\begin{proof}
The assumption that $f(x)\neq -x$ for $x\in\S^{N-1}$ implies that $f$ is homotopic to the identity, through to the homotopy
\[
\Phi_t(x) = \frac{t x + (1-t) f(x)}{|t x + (1-t) f(x)|}\,.
\]
On the other hand, if $F$ has no zero points in $B$ then $f$ is homotopic to a constant function, through the homotopy
\[
\widetilde\Phi_t(x) = \frac{F((1-t)x)}{|F((1-t)x)|}\,.
\]
Since the identity on $\S^{N-1}$ is not homotopic to a constant function, the contradiction shows the existence of some zero points of $F$ in $B$.
\end{proof}

\proofof{Proposition~\ref{mainreduction}}
First of all we notice that, if the claim has been proved for some $\eps>0$, then it is emptily true for every $\eps'>\eps$, with $K(N,\alpha,\eps')=K(N,\alpha,\eps)$. As a consequence, it is sufficient to show the claim for $\eps<\eps_1$. Let then $\eps<\eps_1$ be given, and define $\delta_\eps$ according to Lemma~\ref{firstreduction}, $C=C(N,\alpha)$ according to Proposition~\ref{defex}, and $\eta=\eta(\delta_\eps,\alpha,N)$ according to Lemma~\ref{reduction}.\par

Let now $E$ be any set with $|E|=\omega_N$. If $\delta(E)\geq \delta_\eps$, by Lemma~\ref{reduction} we have $D(E)\geq \eta$, hence
\[
\delta(E) \leq 2\omega_N \leq \frac{2\omega_N}{\sqrt\eta}\, \sqrt{D(E)}\,,
\]
thus~(\ref{sharpest}) holds true with
\[
K(N,\alpha,\eps) = C(N,\alpha) \vee \frac{2\omega_N}{\sqrt\eta}\,.
\]
Assume then now that $\delta(E)<\delta_\eps$. By Lemma~\ref{firstreduction}, we get a set $E'$ such that $|E'|=\omega_N$, satisfying the inclusions $B(1-\eps^2)\subseteq E' \subseteq B(1+\eps^2)$, and such that~(\ref{firstineq}) holds. We can assume that~(\ref{sharpest}) does not hold for $E'$, because otherwise it holds also for $E$ by~(\ref{firstineq}) and the proof is already concluded.\par
For every $z$ such that $|z|\leq \eps/2$, we have the inclusions $B_z(1-\eps)\subseteq E' \subseteq B_z(1+\eps)$, hence we can apply Proposition~\ref{defex} to get a set $E_z$ satisfying~(\ref{secondineq}) and being a $\eps$-nearly spherical set around $B_z$. Putting together~(\ref{firstineq}) and~(\ref{secondineq}), we get that the set $\widetilde E = E_z -z$ satisfies~(\ref{standineq}), hence it completes the proof if the barycenter of $E_z$ is precisely $z$. Therefore, we are reduced to find some $z\in B(\eps/2)$ such that $Bar(z)=z$. Let us set
\[
\xi = \min_{|z|=\eps/2} \big\{ |z-Bar(z) |\big\}\,.
\]
Notice that the minimum exists by continuity thanks to Remark~\ref{remcont}, and it is strictly positive by Lemma~\ref{direction}. We can then define the function $F:\overline B\to \overline B$ as
\[
F(x) = \frac 1\xi\, \Pi\big(\eps x/2-Bar(\eps x/2)\big)\,,
\]
where $\Pi:\R^N\to \overline B(\xi)$ is the projection over $\overline B(\xi)$, that is, $\Pi(y)=y$ if $|y|\leq \xi$, and $\Pi(y) = \xi y/|y|$ otherwise. Notice that the function $F$ is continuous because so are $\Pi$ and $Bar$, by Remark~\ref{remcont}. Take now $x\in\S^{N-1}$: by definition of $\xi$, $F(x)\in \S^{N-1}$, and moreover by Lemma~\ref{direction} $F(x)\neq -x$.\par

We can then apply Lemma~\ref{ramiro} to the function $F$, finding some $x\in B$ such that $F(x)=0$. Calling $z=\eps x/2$, this means that $Bar(z)=z$, thus the proof is concluded.
\end{proof}

\section{Conclusion\label{three}}

This section is devoted to prove Theorem~\mref{Main}. Having reduced ourselves in the previous sections to consider a nearly spherical set with barycenter in $0$, what we have to do is to perform a ``Fuglede-type'' calculation. In other words, Theorem~\mref{Main} comes by putting together Proposition~\ref{mainreduction} from last Section and the following Proposition~\ref{fuglede}.

Given a function $u\in L^2(\partial B)$, we set for $1<\alpha<N$
\begin{equation}\label{semi}
[u]_{\frac{1-\alpha}{2}}^2=\int_{\partial B}\int_{\partial B}\frac{|u(x)-u(y)|^2}{|x-y|^{N-\alpha}}\,d\mathcal H^{N-1}_xd\mathcal H^{N-1}_y.
\end{equation}
It is well known that this semi-norm can be written in terms of the Fourier coefficients $a_{k,i}(u)$ of $u$ with respect to the orthonormal basis of spherical harmonics $y_{k,i}$, where $k\in\mathbb N\cup\{0\}$, $i=1,\dots,N(k)$. In particular, $y_{0,1}=1/\sqrt{N\omega_N}$, $y_{1,i}=x_i/\sqrt{\omega_N}$ for $i=1,\dots,N$. More precisely, we have
\begin{equation}\label{fourier}
\int_{\partial B}\int_{\partial B}\frac{|u(x)-u(y)|^2}{|x-y|^{N-\alpha}}\,d\mathcal H^{N-1}_xd\mathcal H^{N-1}_y=\sum_{k=0}^\infty\sum_{i=1}^{N(k)}\mu^\alpha_ka_{k,i}(u)^2\,, 
\end{equation}
where for all $k\geq0$
\begin{equation}\label{autov}
 \mu_k^\alpha=2^\alpha\,\pi^{\frac{N-1}2}\,\frac{\Gamma(\frac{\alpha-1}2)}{\Gamma(\frac{N-\alpha}2)}\,
 \bigg(\frac{\Gamma(\frac{N-\alpha}2)}{\Gamma(\frac{N-2+\alpha}{2})}-\frac{\Gamma(k+\frac{N-\alpha}2)}{\Gamma(k+\frac{N-2+\alpha}2)}\bigg)\,.
\end{equation}
It is easily checked that the sequence $\mu_k^\alpha$ is bounded from above and strictly increasing in $k$ with $\mu_0^\alpha=0$. Therefore from \eqref{fourier} it follows that
\[
\|u-(u)_{\partial B}\|_{L^2(\partial B)}\approx [u]_{\frac{1-\alpha}{2}}\,,
\]
where $(u)_{\partial B}$ stands for the average of $u$ on $\partial B$.
Moreover, see \cite[Prop. 7.5]{F2M3},
\begin{equation}\label{mu1}
\mu_1^\alpha=\alpha(N+\alpha)\frac{\F(B)}{N\omega_N}\,.
\end{equation}
Next proposition gives a stability estimate for nearly spherical sets. Its proof is based on the argument used in \cite[Th. 1.2]{F} to prove the stability of the isoperimetric inequality for nearly spherical sets, see also \cite[Lemma 5.3]{F2M3}.
\begin{prop}
\label{fuglede}
Let $\alpha\in(1,N)$ be given. There exist positive constants $\eps_0$, $C_0$ and $C_1$, depending only on $N$, such that if $E\subseteq\R^N$ is a measurable set with $|E|=|B|$, with barycenter at the origin and
\[
E=\big\{(1+u(x))\rho x:x\in\partial B, 0<\rho<1\big\}
\]
for some function $u\in L^\infty(\partial B)$ with $\|u\|_{L^\infty(\partial B)}\leq\eps_0$, then
\begin{equation}\label{fuglede00}
\F(B)-\F(E)\geq C_0 \|u\|_{L^2(\partial B)}^2\geq C_1|E\Delta B|^2\,.
\end{equation}
\end{prop}
\begin{proof}
Up to replacing $u$ with $tu$, we may assume that
\[
 E=\big\{(1+tu(x))\rho x:x\in\partial B, 0<\rho<1\big\}\,,
\]
with $\|u\|_{L^\infty(\partial B)}\leq1$ and $t\in(0,\eps_0)$, where $\eps_0\in(0,1/2)$ will be chosen later. Using polar coordinates we may write
\[
\F(E)=\int_{\partial B}d\mathcal H^{N-1}_x\!\int_{\partial B}d\mathcal H^{N-1}_y\!\int_0^{1+tu(x)}dr\!\int_0^{1+tu(y)}\frac{r^{N-1}\rho^{N-1}}{\big(|r-\rho|^2+r\rho|x-y|^2)^{\frac{N-\alpha}{2}}}\,d\rho\,.
\]
Using the identity
\[
2\,\int_0^a\!\int_0^b=\int_0^a\!\int_0^a+\int_0^b\!\int_0^b-\int_a^b\!\int_a^b\,,\qquad a,b\in\R\,,
\]
the previous equality can be rewritten as
\begin{equation}\begin{split}\label{fuglede1}
 \F(E)&=\int_{\partial B}d\mathcal H^{N-1}_x\!\int_{\partial B}d\mathcal H^{N-1}_y\!\int_0^{1+tu(x)}dr\!\int_0^{1+tu(x)}f(|x-y|,r,\rho)\,d\rho
\\
&\qquad\qquad-\frac12\int_{\partial B}d\mathcal H^{N-1}_x\!\int_{\partial B}d\mathcal H^{N-1}_y\!\int_{1+tu(y)}^{1+tu(x)}dr\!\int_{1+tu(y)}^{1+tu(x)}f(|x-y|,r,\rho)\,d\rho\,,
\end{split}\end{equation}
where, for every $r,\rho>0$ and $q\geq0$, we have set
\[
f(q,r,\rho)=\frac{r^{N-1}\rho^{N-1}}{\big(|r-\rho|^2+r\rho\,q^2)^{\frac{N-\alpha}{2}}}
\]
For every $x\in\partial B$, by a change of variable, we get
\begin{align*}
& \int_{\partial B}d\mathcal H^{N-1}_y\!\int_0^{1+tu(x)}dr\!\int_0^{1+tu(x)}f(|x-y|,r,\rho)\,d\rho\\
 &\qquad\qquad=(1+tu(x))^{N+\alpha}\int_{\partial B}d\mathcal H^{N-1}_y\!\int_0^{1}dr\!\int_0^{1}f(|x-y|,r,\rho)\,d\rho
=(1+tu(x))^{N+\alpha}\,\frac{\F(B)}{N\omega_N}\,, 
\end{align*}
where in the last iequality we have used \eqref{fuglede1} with $u=0$. Hence,
\begin{align*}
 \F(E)&=-\frac12\int_{\partial B}d\mathcal H^{N-1}_x\!\int_{\partial B}d\mathcal H^{N-1}_y\!\int_{1+tu(y)}^{1+tu(x)}dr\!\int_{1+t\,u(y)}^{1+t\,u(x)}f(|x-y|,r,\rho)\,d\rho \\
&\qquad\qquad+\frac{\F(B)}{N\omega_N}\,\int_{\partial B}(1+tu)^{N+\alpha}\,d\mathcal H^{N-1}\,.
\end{align*}
 Thus, we obtain
\begin{equation}\label{fuglede0}
\F(B)-\F(E)=\frac{t^2}2g(t)+\frac{\F(B)}{N\omega_N}\,(h(0)-h(t))\,,
\end{equation}
where we have set $h(t)=\int_{\partial B}(1+tu)^{N+\alpha}\,d\mathcal H^{N-1}$ and
\begin{align*}
g(t)=\int_{\partial B}d\mathcal H^{N-1}_x\!\int_{\partial B}d\mathcal H^{N-1}_y\!\int_{u(y)}^{u(x)}dr\!\int_{u(y)}^{u(x)}f(|x-y|,1+tr,1+t\rho)\,d\rho\,.
\end{align*}
Note that $h(0)=N\omega_N=N|E|=\int_{\partial B}(1+tu)^N\,d\mathcal H^{N-1}$. Therefore,
\begin{align*}
h(0)-h(t)&=\int_{\partial B}(1+tu)^N\big(1-(1+tu)^{\alpha}\big)\,d\mathcal H^{N-1}\\
 &\geq -\alpha t\,\int_{\partial B}u\,d\mathcal H^{N-1}-\alpha(2N+\alpha-1)\,\frac{t^2}2\int_{\partial B}u^2\,d\mathcal H^{N-1}-C(N) t^3\|u\|_{L^2}^2\,.
\end{align*}
Using again the assumption $|E|=|B|$ we have that $\int_{\partial B}\big((1+tu)^N-1\big)\,d\mathcal H^{N-1}=0$, which in turn yields
\[
-t\int_{\partial B}u\,d\mathcal H^{N-1}\geq (N-1)\frac{t^2}{2}\int_{\partial B}u^2\,d\mathcal H^{N-1}-C(N) t^3\|u\|_{L^2}^2\,.
\]
Therefore we have
\begin{equation}\label{fuglede2}
h(0)-h(t)\geq-\alpha(N+\alpha)\frac{t^2}{2}\int_{\partial B}u^2\,d\mathcal H^{N-1}-C(N) t^3\|u\|_{L^2}^2\,.
\end{equation}
Concerning $g$, observe that
\begin{equation}\label{fuglede3}
g(0)=\int_{\partial B}\int_{\partial B}\frac{|u(x)-u(y)|^2}{|x-y|^{N-\alpha}}\,d\mathcal H^{N-1}_xd\mathcal H^{N-1}_y
\end{equation}
and that $g$ is a smooth function of $t$ in a neighborhood of $0$. Note also that for $r,\rho\in(-1/2,1/2)$ and $q>0$
\[
\Big|\frac{d}{dt}f(q,1+rt,1+\rho t)\Big|=\Big|r\frac{\partial f}{\partial r}f(q,1+tr,1+t\rho)+\rho\frac{\partial f}{\partial \rho}f(q,1+tr,1+t\rho)\Big|\leq\frac{\widehat C(N)}{q^{N-\alpha}}\,.
\]
Thus, if $0<t<\eps_0$,
\[
|g^\prime(t)|\leq\widehat C(N)\int_{\partial B}\int_{\partial B}\frac{|u(x)-u(y)|^2}{|x-y|^{N-\alpha}}\,d\mathcal H^{N-1}_xd\mathcal H^{N-1}_y\,.
\]
Since $g(t)=g(0)+tg^\prime(\tau)$, with $\tau\in(0,t)$, from the inequality above and \eqref{fuglede3} it follows that for $\eps_0$ sufficiently small 
\[
g(t)\geq(1-t\widehat C(N))\int_{\partial B}\int_{\partial B}\frac{|u(x)-u(y)|^2}{|x-y|^{N-\alpha}}\,d\mathcal H^{N-1}_xd\mathcal H^{N-1}_y\,.
\]
Combining this estimate with \eqref{fuglede0} and \eqref{fuglede2} and recalling \eqref{fourier} and \eqref{mu1} we obtain
\begin{equation}\begin{split}\label{fuglede4}
\F(B)-\F(E)&\geq \frac{t^2}{2}(1-t\widehat C(N))\int_{\partial B}\int_{\partial B}\frac{|u(x)-u(y)|^2}{|x-y|^{N-\alpha}}\,d\mathcal H^{N-1}_xd\mathcal H^{N-1}_y \\
&\qquad\qquad-\frac{\F(B)}{N\omega_N}\alpha(N+\alpha)\frac{t^2}{2}\int_{\partial B}u^2\,d\mathcal H^{N-1}-C(N) t^3\|u\|_{L^2}^2 \\
&=\frac{t^2}{2}\Big[(1-t\widehat C(N))\sum_{k=1}^\infty\sum_{i=1}^{N(k)}\mu^\alpha_ka_{k,i}(u)^2-\mu^{\alpha}_1\|u\|_{L^2}^2-C(N) t\|u\|_{L^2}^2\Big]\,.
\end{split}\end{equation}
Using again the assumption $|E|=|B|$ as above we have
\begin{equation}\label{fuglede5}
|a_{0,1}(u)|=\bigg|\frac{1}{\sqrt{N\omega_N}}\int_{\partial B}u\,d\mathcal H^{N-1}\biggr|\leq C(N)t\|u\|_{L^2}^2\,.
\end{equation}
Similarly, recalling that the barycenter of $E$ is at the origin, hence $\int_{\partial B}x_i(1+tu)^{N-1}\,d\mathcal H^{N-1}=0$ for all $i=1,\dots,N$, we may estimate the first order Fourier coefficients of $u$ as follows
\begin{equation}\label{fuglede6}
|a_{1,i}(u)|=\bigg|\frac{1}{\sqrt{\omega_N}}\int_{\partial B}x_iu\,d\mathcal H^{N-1}\biggr|\leq C(N)t\|u\|_{L^2}^2\,.
\end{equation}
Therefore, from these estimate, taking $\eps_0$ small enough and recalling that $\mu^\alpha_k>\mu^\alpha_1$ for $k\geq2$, we have, using \eqref{fuglede5} and \eqref{fuglede6},
\[\begin{split}
&(1-t\widehat C(N))\sum_{k=1}^\infty\sum_{i=1}^{N(k)}\mu^\alpha_ka_{k,i}(u)^2-\mu^{\alpha}_1\|u\|_{L^2}^2-C(N) t\|u\|_{L^2}^2\\
&\qquad\qquad =(1-t\widehat C(N))\sum_{k=1}^\infty\sum_{i=1}^{N(k)}\mu^\alpha_ka_{k,i}(u)^2-\mu^{\alpha}_1\sum_{k=0}^\infty\sum_{i=1}^{N(k)}a_{k,i}(u)^2-C(N) t\|u\|_{L^2}^2 \\
&\qquad\qquad\geq (1-\eps_0C(N))\sum_{k=2}^\infty\sum_{i=1}^{N(k)}(\mu^\alpha_k-\mu^\alpha_1)a_{k,i}(u)^2-\eps_0C(N)\|u\|_{L^2}^2 \\
& \qquad\qquad\geq C'\sum_{k=2}^\infty\sum_{i=1}^{N(k)}\mu^\alpha_ka_{k,i}(u)^2-\eps_0C(N)\|u\|_{L^2}^2 \geq C'\sum_{k=1}^\infty\sum_{i=1}^{N(k)}\mu^\alpha_ka_{k,i}(u)^2-\eps_0C(N)\|u\|_{L^2}^2\,.
\end{split}\]
From this inequality, \eqref{fuglede4} and \eqref{fuglede5} again, we conclude, assuming $\eps_0$ small enough, 
\begin{align*}
\F(B)-\F(E)&\geq \frac{t^2}{2}\Big[C^\prime\sum_{k=1}^\infty\sum_{i=1}^{N(k)}\mu^\alpha_ka_{k,i}(u)^2-\eps_0C(N)\|u\|_{L^2}^2\Big] \\
& \geq\frac{t^2}{2}\Big[C\sum_{k=1}^\infty\sum_{i=1}^{N(k)}a_{k,i}(u)^2-\eps_0C(N)\|u\|_{L^2}^2\Big]
\geq Ct^2\|u\|_{L^2}^2\,.
\end{align*}
This proves the first inequality in \eqref{fuglede00} with $u$ replaced by $tu$. The second one follows by observing that $\|tu\|_{L^1(B)}\geq C(N)|E\Delta B|$.
\end{proof}
\begin{remark}
We observe that from \cite[Lemma 5.3]{F2M3} we have that for a nearly spherical set $E$, with $|E|=|B|$, and $1<\alpha<N$
\[
\F(B)-\F(E)\leq C\|u\|_{L^2(\partial B)}^2\,,
\]
provided $\eps_0$ is sufficiently small. Thus, by combining this inequality with \eqref{fuglede00} we may conclude that for a nearly spherical set $E$, with $|E|=|B|$ and barycenter at the origin, sufficiently close in $L^\infty$ to the unit ball, the gap $\F(B)-\F(E)$ is equivalent to $\|u\|_{L^2(\partial B)}^2$.\end{remark}

\section*{Acknowledgments}

The work on this paper started during a conference at the ICMS of Edinburgh, which we wish to thank for the hospitality. We are also indebted to Almut Burchard and Ramiro Fuente for some useful conversations on the subject. The authors were partially supported by the Grants PRIN 2015PA5MP7, DFG 1687/1-1 and PRIN 2017TEXA3H. Both authors are members of GNAMPA of INdAM.

\end{document}